\documentclass{amsart}
\usepackage{amsfonts}
\usepackage{amssymb}
\usepackage{color}

\newtheorem{thm}{Theorem}[section]

\newtheorem{lem}[thm]{Lemma}

\newtheorem{prop}[thm]{Proposition}

\theoremstyle{definition}

\theoremstyle{remark}
\newtheorem{rem}{Remark}[section]

\numberwithin{equation}{section}

\begin{document}

\title[On a homotopy formula  for generalized steady Stokes' operators
]
{On a homotopy formula for generalized steady Stokes' operators, \\ associated with the de Rham complex}

\author{U. Kiseleva}
\email{ulita.kiseleva@gmail.com} 

\author{A.A. Shlapunov}
\email{ashlapunov@sfu-kras.ru}  
 

\address[Ulita Kiseleva, Alexander Shlapunov]{Siberian Federal University
                                                 \\
         pr. Svobodnyi 79
                                                 \\
         660041 Krasnoyarsk
                                                 \\
         Russia}



\begin{abstract}
We construct left, right and bilateral fundamental solutions for 
generalized steady Stokes' operators $S$ with smooth coefficients coefficients, associated 
with the de Rham complex of differentials on differential forms over a domain $X$ in ${\mathbb R}^n$. 
The investigated operators are Douglis-Nirenberg elliptic under 
reasonable assumptions. As an immediate corollary we produce a homotopy formula for 
regular solutions to this operator.
\end{abstract}

\maketitle

\section*{Introduction}
\label{s.Int}

The crucial components providing many results for solutions to Partial Differential Equations are the following: 
a version of the so-called Unique 
Continuation Property for solutions, regularity theorems and existence of a bilateral regular 
fundamental solution/parametrix for  the investigated Differential Operator. Both 
Douglis-Nirenberg elliptic systems and elliptical parabolic systems reputedly have the mentioned above 
properties up to some extent. However, the construction of a fundamental solution, may be a rather difficult task, 
except some special cases. For instance, the Fourier and Laplace transform  give tools to 
do it for operators with constant coefficients, see \cite{Mal56}, \cite{Mal63}, \cite{Vla}.

	In the present paper we investigate 
generalized steady Stokes' operators $S$ with smooth coefficients coefficients, associated 
with the de Rham complex of differentials on differential forms over a domain $X$ 
in ${\mathbb R}^n$,  introduced in \cite{ShPMi}. 
The investigated operators are Douglis-Nirenberg elliptic under 
reasonable assumptions and have some properties similar to the classical Stokes operators, see 
\cite{Mas97}, \cite{MeShT}, 
\cite{PaSh21}, \cite{ShPo}, \cite{ShTaNS}, \cite{Tema79}. We
	describe the general form of solutions to these operators. Moreover, 
	for a particular case, where only one diagonal element of Stokes' matrix is non-zero and has real 
	analytic entries, we 
construct (in an explicit form) suitable  bilateral fundamental solutions 
for them. As an immediate corollary we produce a homotopy formula for 
regular solutions to this type of operators.

\section{Preliminaries}

Let ${\mathbb R}^n$, $n \geq 2$, be the $n$-dimensional Euclidean space with the coordinates  
$x=(x_1, \dots , x_n)$ and let  $D \subset {\mathbb R}^n$ be a bounded domain (open connected set). As 
usual, denote by  $\overline{D}$ the closure of $D$, and by  $\partial D$ its boundary. 

For $s \in {\mathbb Z}_+$ we denote by $C^s(D)$ and $C^s(\overline D)$  the spaces of all $s$ times 
continuously differentiable functions on $D$ and $\overline D$, respectively;  
$C^\infty (D) = \cap_{s\in {\mathbb Z}_+} C^s (D)$. We endow the space $C^s(D)$ with the standard 
Fr\'echet  topology of the uniform convergence  on compact 
subsets of $D$ with all the partial derivatives up to order $s$. Let also 
$C^\infty_0 (D)$ be the set of smooth functions with compact support in $D$. 

Let $\Lambda^q$ stand for the (trivial) vector  bundle of the exterior differential forms of degree $q$ on 
${\mathbb R}^n$. As it is known, the rang of the bundle $\Lambda ^q$ equals to the binomial coefficient 
$k_q = \left( \begin{array}{ll} n \\ q \\ \end{array} \right)$. 

Recall that a differential form $u$ of a degree $q$, 
$0\leq q \leq n$, of some topological space ${\mathfrak C} (D,\Lambda^q)$ on the domain $D$ 
is given by 
$$
u(x) = \sum_{\sharp I = q} u _I(x) dx_I,
$$
where $I= (i_1, \dots i_q)$, $dx_I = dx_{i_1} \wedge \dots \wedge dx_{i_q}$, $1\leq i_j \leq n$,
$\wedge$ is the exterior product of differential forms, providing the relation $dx_i\wedge dx_j = 
- dx_j\wedge dx_i $ for differentials $dx_i$, and the coefficients $u_I$ belongs to  ${\mathfrak C} (D)$, 
 see for instance, \cite{deRham}, \cite[Ch. 6]{WRL},  
The class will be endowed 
with the topology induced from ${\mathfrak C} (D)$ component-wise.

Thus, let $\{ d_q, \Lambda^q\}$ be the de Rham complex of exterior differentials on differential forms on 
${\mathbb R}^n$, see for instance, \cite{deRham}, \cite[Ch. 6]{WRL},  
\begin{equation} \label{eq.deRham}
0\rightarrow 
C^\infty({\mathbb R}^n,\Lambda^0)\stackrel{d_0}{\rightarrow}C^\infty({\mathbb R}^n, \Lambda^1)
\rightarrow \dots 
\stackrel{d_{n-1}}{\rightarrow} C^\infty({\mathbb R}^n, \Lambda^n) \rightarrow 0, 
\end{equation}
The de Rham differentials $d_q$, 
$$
d_q u = \sum_{j=1}^n  \sum_{\sharp I = q} \frac{\partial u _I}{\partial x_j} dx_j \wedge dx_I,
$$
 satisfy familiar relations
\begin{equation} \label{eq.deRham.0}
d_q = 0 \mbox{ if } q<0 \mbox{ or } q\geq n, \, 
d_q \, d_{q-1} = 0 .  
\end{equation}

Let $\star$ be the $\star$-Hodge operator,  see for instance, \cite{deRham}, \cite[Ch. 6]{WRL}, 
mapping $q$-forms to $(n-q)$-forms in such a way that 
for $q$-forms $u,v$ we have 
$$
u \wedge \star v = \sum_{\sharp I = q} v_I u_I dx .
$$

Let $Y$ be a  measurable subset in ${\mathbb R}^n $ and let $L^2 (Y)$ be the standard Lebesgue space
with the inner product  
$$
(u,v)_{L^2 (Y)} = \int_Y v (x) u(x) dx.
$$
The operator $\star$ may be used to define the inner product on the space $L^2 (Y,\Lambda^q)$ 
of differential forms of the degree $q$, $0\leq q \leq n$, with $L^2 (Y)$ coefficients:
$$
(u,v)_{L^2 (Y, \Lambda^q)} = \int_Y u(x)\wedge \star v (x).
$$

Denote by $d^*_q$ the formal adjoint differential operator for $d_q$:
$$
(d_q u,v)_{L^2 ({\mathbb R}^n, \Lambda^{q+1})} = ( u, d^*_qv)_{L^2 ({\mathbb R}^n, \Lambda^{q})} 
\mbox{ for all } v \in  C^\infty_0 ({\mathbb R}^n, \Lambda^{q+1}).
$$ 
As it is known, see \cite[\S 2.5.2]{Tark35}, \cite{deRham}, \cite[Ch. 6]{WRL}, 
$d^*_q v =  (-1)^{nq +1}\star d_{n-q-1} \star v$ for a $(q+1)$-form $v$. 

Then Stokes integration formula provides the (first) Green formula for the differential 
operator $d_q$ in any Lipschitz domain $D$:
\begin{equation} \label{eq.Green.d}
\int_{\partial D} u \wedge 
\star v =  (d_q u,v)_{L^2 (D, \Lambda^{q+1})} - ( u, d^*_qv)_{L^2 (D, \Lambda^{q})} 
\end{equation}
for all $u \in  C^\infty (\overline D, \Lambda^{q})$, $v \in  C^\infty (\overline D, \Lambda^{q+1})$, \cite[\S 2.5.2]{Tark35}.


Next, let   
\begin{equation} \label{eq.Laplace.0}
\Delta_{q} = d_q^* d_q + d_{q-1} d^*_{q-1}
\end{equation}
stand for the Hodge Laplacians of the de Rham complex,  
see, for instance, \cite[\S 2.5.2]{Tark35}, \cite{deRham}, \cite[Ch. 6]{WRL}.
The differential 
operators $\Delta_{q}$ are strongly elliptic, formally self-adjoint 
and coincide with the (minus) matrix Laplace operator, applied to a $q$-form $u$ coefficient-wise:
\begin{equation} \label{eq.Laplace.1}
\Delta_{q} u = - \sum_{\sharp I = q} (\Delta u _I) dx_I , \, 0\leq q \leq n.
\end{equation}

By \eqref{eq.deRham.0} we easily obtain
\begin{equation} \label{eq.deRham.1}
 d^*_{q-1} d^*_q = 0, \,
d_q \Delta_q =  \Delta_{q+1} d_q = d_q d_q^* d_q , \, d^* _{q-1} \Delta_q =  \Delta_{q-1} d^*_{q-1}  = 
d^*_{q-1} d_{q-1} d^*_{q-1} .
\end{equation}

If we treat the operators $d_j, d^*_j$ as matrix differential operators, the we may introduce 
Lam\'e type operators:
$$
\Delta_{q,\mu} = d_q^* {\mathcal M}_q d_q + d_{q-1} \tilde {\mathcal M}_q d^*_{q-1}  ,
$$
for some pair $\mu_q = ({\mathcal M}_q$, $\tilde {\mathcal M}_q)$ of functional matrices  with smooth entries on 
the closure $\overline X$ of a domain $X \subset {\mathbb R}^n$. If these matrices 
are self-adjoint and positive on $\overline X$,  then the differential 
operators $\Delta_{q,\mu}$ are strongly elliptic, formally self-adjoint and hypoelliptic on $X$. If, in addition, 
the entries of the matrices ${\mathcal M}_q$, $\tilde {\mathcal M}_q$ are real 
analytic then solution to the operators $\Delta_{q,\mu}$ are real analytic  by Petrovskii theorem.

Similarly to \eqref{eq.deRham.1}, we have 
\begin{equation} \label{eq.deRham.1mu}
d_q^* {\mathcal M}_q d_q \Delta_{q,\mu} =  \Delta_{q,\mu} d_q^* {\mathcal M}_q d_q , \, 
d _{q-1} \tilde {\mathcal M}_q d^* _{q-1} \Delta_{q,\mu} =  \Delta_{q,\mu} d _{q-1} \tilde  {\mathcal M}_q d^*_{q} , 
\end{equation}

In the framework of theory differential forms, the multiplication ${\mathcal M} u$ for a self-adjoint matrix 
${\mathcal M}$ and $q$-form $u$ may be organized as follows. We identify ${\mathcal M}$ with a set $q$-differential 
forms $\{ {\mathcal M}^{(I)} \}_{\sharp I = q}$, satisfying ${\mathcal M}_{J}^{(I)} = 
{\mathcal M}_{I}^{(J)}$, and then 
\begin{equation} \label{eq.multip}
{\mathcal M} u = \sum_{\sharp I = q} (\star (u \wedge \star {\mathcal M}^{(I)} ))dx_I.
\end{equation}
In this way,  formulae \eqref{eq.Green.d}, \eqref{eq.multip} induce the (first) 
Green formula for the differential 
operator $\Delta_{q,\mu}$ in a Lipschitz domain $D$, \cite[\S 2.4.2]{Tark35}:
\begin{equation} \label{eq.Green.Delta1} 
\int_{\partial D} {\mathcal G}_{\Delta_{q,\mu}} (v,u)= 
 (\Delta_{q,\mu} u,v)_{L^2 (D, \Lambda^{q})} - ( u,  \Delta_{q,\mu} v)_{L^2 (D, \Lambda^{q})} 
\mbox{ for all } u,v \in  C^\infty (\overline D, \Lambda^{q}),
\end{equation}
where  ${\mathcal G}_{\Delta_{q,\mu}} (\cdot,\cdot)$ is
the Green operator for $\Delta_{q,\mu}$ that is  given by
\begin{equation}
\label{eq.GDelta}
v\wedge \star ({\mathcal M}_q d_q u ) 
+ (\tilde  {\mathcal M}_q d^*_{q-1} u) \wedge 
\star v  
 -  u\wedge \star ({\mathcal M}_q d_q v ) 
- (\tilde  {\mathcal M}_q d^*_{q-1} v) \wedge 
\star u  . 
\end{equation}

\section{Stokes' operators}

Consider  the Stokes'  type operator for forms of degrees $0$ and $1$:  
$$
S_{1,\mu} = \left( 
\begin{array}{lll}
\Delta_{1,\mu} & d_0 \\
d_0^* & 0 \\
\end{array}
\right) .
$$
This gives the classical Stokes' operator if ${\mathcal M}_1$ and 
$\tilde {\mathcal M}_1$ are unit matrices of the corresponding dimensions, \cite{Mas97}, 
\cite{Tema79}, playing an essential role in Hydrodynamics.

For arbitrary $q$, $1\leq q \leq n$, generalized 
Stokes' operators $S_{q,\mu}$ can be defined as three-diagonal $((q+1)\times (q+1))$-block 
matrix, see \cite{ShPMi}, with the following block-entries:
$$
S_{q,\mu} ^{j,j} = \Delta_{q-j+1,\mu},   \, 1 \leq j \leq q+1, \, 
S_{q,\mu} ^{i,i+1} = d_{q-i}  , \, S_{q,\mu} ^{i+1,i} = d^*_{q-i}  \, 1 \leq i \leq q,
$$
$$
S_{q,\mu} ^{i,j} =0, \, 1 \leq i,j \leq q+1, i\ne j, i\ne j+1, \, j\ne i+1, 
$$
or, in the matrix form,
$$
S_{q,\mu} = \left( 
\begin{array}{llllllllll}
\Delta _{q,\mu_q}  & d_{q-1} & 0 & \dots & \dots & \dots & \dots & \dots & 0 \\
d^*_{q-1} & \Delta _{q-1,\mu} & d_{q-2} & 0 & \dots & \dots & \dots & \dots & 0 \\
0 & d^*_{q-2} & \Delta _{q-2,\mu} & d_{q-3} & 0 & \dots & \dots & \dots & 0 \\
\dots & \dots & \dots  & \dots  & \dots & \dots & \dots & \dots & \dots \\
\dots & \dots & \dots  & \dots  & \dots & \dots & \dots & \dots & \dots \\
0 &  \dots  &  \dots  &  \dots & 0 &  d_{2}^* & \Delta _{2,\mu}  & d_1 & 0 \\
0 &  \dots & 0 &  0 & \dots & 0 & d_1^*& \Delta _{1,\mu} & d_0 \\
0 &  \dots & 0 &  0 &  & 0 & 0 & d^*_0 & \Delta _{0,\mu} \\
\end{array}
\right) 
 $$
where $\mu = (\mu_q, \dots \mu_q)$, and the pairs $\mu_j = ({\mathcal M}_j, \tilde {\mathcal M}_j)$, $0\leq j \leq j_0<q$ can be zeroes. 
The second order differential operator $S_{q,\mu}$ maps the space 
${\mathbf C}^\infty_q (X) = \oplus_{j=0}^q C^\infty (X, \Lambda ^{q-j})$ to itself. 
We tacitly assume that $X={\mathbb R}^n$ if the coefficients of the operator $S_{q,\mu}$ are constant.  

Formulae \eqref{eq.Green.d}, \eqref{eq.Green.Delta1} induce the (first) 
Green formula for the differential 
operator $S_{q,\mu}$ in a Lipschitz domain $D$:
\begin{equation} \label{eq.Green.Sqmu1}
\int_{\partial D} {\mathcal G}_{S_{q,\mu}} (v,u)
= (S_{q,\mu} u,v)_{{\mathbf L}^2_q (D)} - ( u,  S_{q,\mu} v)_{{\mathbf L}^2_q (D)} 
\end{equation}
for all $u = (u_q, \dots u_0)$, $v=(v_q, \dots v_0)\in  {\mathbf C}_{q}^\infty (\overline D, \Lambda^{q})$,
where 
$${\mathbf L}^2_q (D) = \oplus_{j=0}^q L^2 (D, \Lambda ^{q-j}), \, 
{\mathbf C}^\infty_{q} (\overline D) = \oplus_{j=0}^q C^\infty (\overline D, \Lambda ^{q-j}),
$$
 and ${\mathcal G}_{S_{q,\mu}} (\cdot,\cdot)$ is 
the Green operator for $S_{q,\mu}$, that is  given by
\begin{equation} \label{eq.GS}
{\mathcal G}_{S_{q,\mu}} (v,u) = 
 \sum_{j=0}^{q-1} \big( u_{j} \wedge \star v_{j+1} - v_{j} \wedge \star u_{j+1} + 
G_{\Delta_{j,\mu}} (v_j,u_j)\big) + G_{\Delta_{q,\mu}} (v_q,u_q) . 
\end{equation} 
 
Obviosly, $S_{q,\mu}$ is (Petrovskii) 
elliptic if all the matrices ${\mathcal M}_j, \tilde {\mathcal M}_j$, 
$0\leq j \leq q$, are positive. 
It was shown in \cite{ShPMi} that $S_{q,\mu}$ is Douglis-Nirenberg elliptic if 
matrices ${\mathcal M}_q, \tilde {\mathcal M}_q$ are positive on $X$ (cf. \cite{Mas97} for the classical 
Stokes' operator). 

Thus, if all the matrices ${\mathcal M}_j, \tilde {\mathcal M}_j$, 
$0\leq j \leq q$, are positive, the standard approximation theorems for Petrovskii 
elliptic operators are still valid for $S_{q,\mu}$. So, we
are interested in the case where $S_{q,\mu}$ is Douglis-Nirenberg elliptic, only.  

As it is known, regularity theorems for Douglis-Nirenberg elliptic operators are similar to 
the Petrovskii elliptic operators, \cite[Ch. 9]{WRL}. 
Using the specific structure, we may show  this fact for  Stokes' operator $S_{q,\mu}$ directly, obtaining 
additional important information on its solutions. 

With this purpose, 
 let ${\mathcal S} _{S_{q,\mu}}(D)$ be the set of all the generalized 
solutions to the  equation
\begin{equation} \label{eq.sol}
S_{q,\mu} u = 0 \mbox{ in } D.
\end{equation}

\begin{prop} \label{p.0} 
Let $1\leq j_0\leq q\leq n$ and   $ \Delta _{j,\mu} = 0$ for all $0\leq j \leq j_0-1$. 
If the self-adjoint matrices ${\mathcal M}_j, \tilde {\mathcal M}_j$, $j_0\leq j\leq q$, are positive and $
C^\infty$-smooth on $\overline X$, then any solution $u = (u_q, \dots u_0) \in {\mathcal S} _{S_{q,\mu}}(D)$
belongs to ${\mathbf C}^\infty_{q} (D)$; 
besides entries of $u_j $ are harmonic for $0\leq j \leq j_0-2$. Moreover, if $j_0>1$ and ${\mathcal M}_j$ and 
$\tilde {\mathcal M}_j$ are real analytic for all $j_0\leq j \leq q$ then $u_q, u_{q-1}, 
\dots u_{j_0-1}$ are real analytic in $D$, too. 
In the exceptional case $j_0 =1$, the function $u_0$ is harmonic and $u_j$, $1\leq j\leq q$, are smooth in $D$;
$u_1, u_2, \dots u_q $ are real analytic, if ${\mathcal M}_1, \dots {\mathcal M}_q$, $\tilde {\mathcal M}_2, \dots 
\tilde {\mathcal M}_q$  are real analytic in $D$.
\end{prop}

\begin{proof} Indeed, if $q \geq 2$ and $u = (u_q, \dots u_0)\in {\mathcal S} _{S_{q,\mu}}(D)$ then 
\begin{equation} \label{eq.Sq.q}
\left\{ 
\begin{array}{lllll}
d_0^* u_1 = 0 , \, \Delta_{q,\mu} u_q + d_{q-1} u_{q-1} = 0, \\
d^*_{j-1} u_j + \Delta_{j-1,\mu} u_{j-1}+ d_{j-2} u_{j-2} = 0, \, j_0+1\leq j\leq q, \\
d^*_{j-1} u_j + d_{j-2} u_{j-2} = 0, \, 2 \leq j\leq j_0, \\
\end{array}
\right.
\end{equation}
and hence, by \eqref{eq.deRham}, we have in $D$:
\begin{equation} \label{eq.Sq.q.c1}
\left\{ 
\begin{array}{lllll}
d_0^* u_1 = 0 , \, d_{j-1}d^*_{j-1} u_j  = 0, \,  d^*_{j-2} d_{j-2} u_{j-2} = 0, \, 2 \leq j\leq j_0, \\ 
d^*_{q} d_{q} d^*_{q} {\mathcal M}_{q} d_{q} u_q  = 0, 
\\ 
d_{q-1}  d^*_{q-1}d_{q-1}  \tilde {\mathcal M}_{q} d^*_{q-1} u_q + 
d_{q-1} d_{q-1} ^* d_{q-1} u_{q-1} = 0, \\
d_{j-1}d^*_{j-1} u_j + d_{j-1} d^*_{j-1} {\mathcal M}_{j-1} d_{j-1} u_{j-1} = 0, \, j_0+1\leq j\leq q, 
\\
d^*_{j-2} d_{j-2}\tilde {\mathcal M}_{j-1} d^*_{j-2} u_{j-1} + d^*_{j-2} d_{j-2} u_{j-2} = 0, \, j_0+1\leq j\leq q.
\end{array}
\right.
\end{equation}
In particular, the first equation in \eqref{eq.Sq.q.c1} yields 
$$
(d^*_{j} d_{j}  +  d_{j-1}d^*_{j-1}) u_j  = \Delta_j u_j= 0, \,  0 \leq j\leq j_0-2,
$$
i.e. $u_0, \dots u_{j_0-2}$ have harmonic coefficients in $D$, i.e. they are real analytic there.  

Next, \eqref{eq.deRham} and \eqref{eq.Sq.q} imply the following identities in $D$:
\begin{equation} \label{eq.Sq.q.c2}
\left\{ 
\begin{array}{lllll}
\Delta_q \Delta_{q,\mu} u_q +  d_{q-1} d^*_{q-1} d_{q-1} u_{q-1}= 0, \\
\Delta_j \Delta_{j,\mu} u_j + d^*_{j} d_{j}d^*_{j} u_{j+1} +  d_{j-1} 
d^*_{j-1} d_{j-1} u_{j-1} = 0,  \, j_0+1 \leq j \leq q-1, \\ 
\Delta_{j_0} (d^*_{j_0} {\mathcal M}_{j_0} d_{j_0} + d_{j_0-1} d^*_{j_0-1}) u_{j_0} + 
d^*_{j_0} d_{j_0} d^*_{j_0} u_{j_0+1} = 0   .
\end{array}
\right.
\end{equation}
It is easy to see that the principal symbol of the fourth order of 
system \eqref{eq.Sq.q.c2} is fourth order with respect to forms $u_j$, $j_0\leq j \leq q$, is non-degenerate. 
Hense the system is (Petrovskii) elliptic and 
by the elliptic regularity, these forms  belong to $C^\infty (D,\Lambda^j)$.  If the coefficients of this 
system are real analytic, 
then, by Petrovskii Theorem, the coefficients of the forms $u_{j_0}, \dots u_q$ are real analytic, too.

In addition, for $j=j_0+1$, 
the last equation in \eqref{eq.Sq.q.c1} means that  the form $u_{j_0-1}$ is a solution to the second 
order (Petrovskii) elliptic system of equations 
with the  form $u_{j_0}$ having the properties discussed above: 
$$
\Delta_{j_0-1} u_{j_0-1} = -  d^*_{j_0-1}  d_{j_0-1} \tilde {\mathcal M}_{j_0} d_{j_0-1}^* u_{j_0}.
$$
Thus, the coefficients of the form $u_{j_0-1}$  are smooth in $D$ if $\tilde {\mathcal M}_{q-1}$ is smooth; 
 they are  real analytic in $D$ if $\tilde {\mathcal M}_{j_0}$ is real analytic.

It is left to consider the exceptional case $j_0=1$.

If $j_0=q=1$ then \eqref{eq.deRham} and \eqref{eq.Sq.q} imply 
\begin{equation*} 
\left\{ 
\begin{array}{lllll}
d_0^* u_1 = 0 , \,
0= d_0^* d^*_1 {\mathcal M}_1 d_1 u_1 +  d_0^* d_0 \tilde {\mathcal M}_1 d_0^* u_1  + d_0^* d_0 u_0 = d_0^* d_0 u_0 , 
\\
0= d_1 d^*_1 {\mathcal M}_1 d_1  u_1  + d_1 d_0 \tilde {\mathcal M}_1 d_0^*u_1  + d_1 d_0 u_0 = 
d_1 d^*_1 {\mathcal M}_1 d_1 .
\end{array}
\right.
\end{equation*}
In particular, $u_0$ is harmonic and $u_1$ satisfies 
the fourth order (Petrovskii) elliptic system
$$
\Delta_1 (d^*_1 {\mathcal M}_1 d_1 + d_0 d_0^*) u_1 = 0  \mbox{ in } D.
$$
If the coefficients of this system are real analytic, then by Petrovskii Theorem, the coefficients 
of the form $u_1$ are real analytic, too.  

Finally, if $q\geq 2$,  $j_0=1$ then the last equation in \eqref{eq.Sq.q.c1}  with $j=j_0+1 = 2$ 
mean that $d_0^*  d_0 u_0 = 0$ in $D$ because $d_0^* u_1 = 0 $. In particular, $u_0$ is harmonic (and real analytic) 
and  besides, \eqref{eq.Sq.q} yields that $u_1, u_2, \dots u_q$ satisfy the second order Petrovskii elliptic system 
of equations:
\begin{equation*} 
\left\{ 
\begin{array}{lll}
\Delta_{2,\mu} u_2 + d_{1} u_{1} = 0  \mbox{ in } D \mbox{ if }  q=2,\\
  d_1^* u_2+ (d_1^* {\mathcal M}_1 d_1   + d_0 d^*_0 ) u_1 = -d_0 u_0  \mbox{ in } D, \\
\end{array}
\right.
\end{equation*}
\begin{equation*} 
\left\{ 
\begin{array}{lll}
\Delta_{q,\mu} u_q + d_{q-1} u_{q-1} = 0  \mbox{ in } D  \mbox{ if }  q\geq 3, \\
 d_{j-1}^* u_{j} +
\Delta_{j-1,\mu} u_{j-1} + d_{j-2} u_{j-2} = 0  \mbox{ in } D, 3 \leq j \leq q,\\
 d_1^* u_2+ (d_1^* {\mathcal M}_1 d_1  + d_0 d^*_0 ) u_1 = -d_0 u_0  \mbox{ in } D, \\
\end{array}
\right.
\end{equation*}
Hence $u_j \in C^\infty (D,\Lambda^j)$, $0\leq j \leq q$. But $u_0$ is a real analytic function and therefore, 
 if the coefficients of this system are real analytic, then by Petrovskii Theorem, the 
coefficients of the forms $u_1,u_2, \dots u_q$ are real analytic, too.  
\end{proof}

\section{
A homotopy formula}

It is known that, similarly to 
the Petrovskii elliptic operators, the Douglis-Nirenberg elliptic operators admit
parametrices and fundamental solutions,  \cite[Ch. 8]{WRL}. 
Using the specific structure, we may show  this fact for  Stokes' operator $S_{q,\mu}$ in a direct and 
constructive way. 

We begin with simplest case. With this purpose, we note that for each $p\in \mathbb N$
formula \eqref{eq.Laplace.1} yields that the powers $\Delta^p_q$ of Laplacians $\Delta_q$ admit bilateral fundamental solutions
$\Phi_{\Delta_q}^p$ on ${\mathbb R}^n$ that are given by 
$$
\Phi_{\Delta_q^p} u =  \sum_{\sharp I = q} (g_p u _I) dx_I , \, (g_p u_I)(x) = \int_{{\mathbb R}^n} g_p(x-y) u_I (y) dy, \, u \in C^\infty_0 ({\mathbb R}^n, \Lambda^q),
$$
where $g_p$ is the standard fundamental solution to the Laplace operator $\Delta^p$ in  ${\mathbb R}^n$:
$$
g_p(x) = 
\left\{ 
\begin{array}{lll} c_{n,p} |x|^{2p-n}\ln |x|, & p-n/2 \in {\mathbb Z}_{\geq 0}, \\ [.1cm]
 c_{n,p} \, |x|^{2p-n}, & p-n/2 \not \in {\mathbb Z}_{\geq 0},
\end{array}
\right. 
$$
with well known constants $c_{n,p}$, depending on the power $p$  and the dimension $n$. For $p=1$ 
we wrtie  $\Phi_q $ instead of $\Phi_{\Delta_q}$.


\begin{thm} \label{t.const} 
Let $q\geq 1$ and  $ \Delta _{j,\mu} = 0$ for all $0\leq j \leq q-1$. If ${\mathcal M}_q = 
\lambda_q I_{k _{q+1}}$, 
$\widetilde {\mathcal M}_q =  \tilde \lambda_q I_{k_{q-1}}$ with some positive numbers  
$\lambda_q$ and  $\tilde \lambda_q$,  then

1) a bilateral fundamental solution to $S_{q,\mu}$ is given by 
the three-diagonal $((q+1)\times (q+1))$-block matrix $\Psi^{(r)}_{q,\mu}$ with the following block-entries:
$$
\Psi_{q ,\mu} ^{1,1} = \lambda_q^{-1}
d^*_{q} d_{q} \Phi_{\Delta_q^2}   , \,  \Psi_{q ,\mu} ^{1,2}  = d_{q-1} \Phi_{q-1},
$$
$$
\Psi_{q ,\mu} ^{2,1} =  
d_{q-1}^* \Phi_{q},\, \Psi_{q ,\mu} ^{2,2} =-  \tilde \lambda_q d_{q-1}^*d_{q-1}  \Phi_{q-1} , 
$$
and, for $q\geq 2$,
\begin{equation} \label{eq.q.geq.2.1}
\Psi_{q,\mu} ^{j,j+1} = d_{q-j} \Phi_{q-j}, \, 
\Psi_{q,\mu} ^{j+1,j} = d^*_{q-j} \Phi_{q-j+1}  ,  \, 2 \leq j \leq q,
\end{equation}
\begin{equation}
\label{eq.q.geq.2.2}
\Psi_{q,\mu} ^{i,j} =0, \, 3 \leq i,j \leq q+1, i\ne j+1, \, j\ne i+1;
\end{equation}


2) if $q=1$ then  
$\Psi_{q,\mu} ^{2,2}$ coincides with $(- \tilde \lambda_1)$;

\end{thm}

\begin{proof} Calculating directly, we obtain 
\begin{equation}\label{eq.D2q}
\Delta g_2(x) = g_1(x) +c_n, \, \,
 d^*_{q-1} \Delta_q \Phi_{\Delta^2_q} = 
 d^*_{q-1}  \Phi_{q},
\end{equation} 
where  $c_n$ equals to zero in all the dimensions, except $n=2$.

If $q=1$ then,  
$d_{q-1}^* d_{q-1} \Phi_{q-1} =
d_{0}^* d_{0} \Phi_{0} = 
\Phi_{q-1} d_{q-1}^* d_{q-1}  =
\Phi_{0} d_{0}^* d_{0} = 
I$. Hence
\begin{equation} \label{eq.fund.ide.11.const}
\Psi_{1,\mu}  = \left( 
\begin{array}{lll}
\lambda_1^{-1} d^*_{1} d_{1} \Phi_{\Delta_1^2}    & 
d_0 \Phi_0 \\
 d_0^* \Phi_{1}  & - \tilde \lambda_1 \\
\end{array}
\right)  = 
 \left( 
\begin{array}{lll}
\lambda_1^{-1} d^*_{1} d_{1} \Phi_{\Delta_1^2}    & 
d_0 \Phi_0 \\
 d_0^* \Delta_1 \Phi_{\Delta_1^2}  & - \tilde \lambda_1 \\
\end{array}
\right).
\end{equation}

Then
\eqref{eq.deRham.1}, 
\eqref{eq.deRham.1mu}, 
\eqref{eq.D2q}, we have  $\Delta_{1,\mu} d_1^* d_1 = \lambda_1  \Delta_{1} d_1^* d_1$,  and  
$$
\Delta_{1,\mu} \Psi_{1,\mu} ^{1,1} + d_{0} \Psi_{1,\mu} ^{2,1} = 
\Delta_{1,\mu}
\lambda_1^{-1} d_1^* d_1   \Phi_{\Delta_1^2}  + d_0
 d_0^* \Delta_1 \Phi_{\Delta_1^2} =\Delta_1^2 \Phi_{\Delta_1^2} 
= I, $$
$$
d^*_0 d_0 \Phi_{0}  = I, \, 
d_0^*  \lambda_1^{-1} d_1^*  d_1   \Phi_{\Delta_1^2} =  0,
$$
$$
\Delta _{1,\mu} 
 d_0 \Phi_{0} 
 -  d_0 \tilde {\lambda}_1  = d_0 \tilde {\lambda}_1 d^*_0 d_0 \Phi_{0} 
- d_0 \tilde {\lambda}_1 = d_0 \tilde {\lambda}_1  - d_0 \tilde {\lambda}_1  = 0,
$$
i.e.  
$\Psi_{1,\mu}$ is a right fundamental solution to  $S_{1,\mu}$. 


If $q\geq 2$ then, 
by \eqref{eq.deRham.1}, 
\eqref{eq.deRham.1mu}, 
\eqref{eq.D2q}, 
$\Delta_{q,\mu} d_q^* d_q = \lambda_q  \Delta_{q} d_q^* d_q$, and  the multiplications of the first line of 
$S_{q,\mu}$ to the first three columns of 
$\Psi_{q,\mu}$ give us
$$
\Delta_{q,\mu} \Psi_{q,\mu} ^{1,1} + d_{q-1} \Psi_{q,\mu} ^{2,1} + 0\cdot 0 = 
$$
$$
\Delta_{q,\mu}
\lambda_q^{-1} d_q^* d_q   \Phi_{\Delta_q^2}  +
d_{q-1}
 d_{q-1}^* \Delta_q \Phi_{\Delta_q^2} =\Delta_q^2 \Phi_{\Delta_q^2} 
 =I_{k_q}, 
$$
$$
\Delta_{q,\mu}  d_{q-1}\Phi_{q-1} +   d_{q-1} \Psi_{q,\mu} ^{2,2} + 
0\cdot \Phi_{q-2}  d ^{*}_{q-2} + 0 \cdot 0 = 
$$
$$
\tilde \lambda_q d_{q-1}  d_{q-1}^* d_{q-1}  \Phi_{q-1} - \tilde \lambda_q
 d_{q-1} d_{q-1}^*
 d_{q-1} \Phi_{q-1}  + 
0\cdot \Phi_{q-2}  d ^{*}_{q-2} = 0.
$$
The multiplications 
of the first row of $S_{q,\mu}$ to the other columns of $\Psi_{q,\mu}$ equal, obviously, to zero.

The multiplications 
of the second line of $S_{q,\mu}$ to the first four columns of $\Psi_{q,\mu}$ give us
$$
 d^*_{q-1} \Psi_{q} ^{1,1}  + 0 \cdot \Psi_{q} ^{2,1}+ 0 \cdot 0  = 
 d^*_{q-1}  
\lambda_{q}^{-1} d_q^*  d_q \Phi_{\Delta_q^2}   =0 , 
$$
$$
d^*_{q-1}  d_{q-1}\Phi_{q-1} + 0\cdot \Psi_{q} ^{2,2} + d_{q-2} 
d^*_{q-2} \Phi_{q-1}   + 0\cdot 0 
= \Delta_{q-1}\Phi_{q-1} =I_{k_{q-1}}, 
$$
$$
d_{q-1}^* \cdot 0  + 0 \cdot  d_{q-2} \Phi_{q-2} + d_{q-2}  \cdot 0 + 0 \cdot 0= 0, 
$$
$$
d_{q-1}^* \cdot 0  + 0 \cdot 0 + d_{q-2} d_{q-3} \Phi_{q-3} + 0 \cdot 0 = 0, 
$$
and the multiplications 
of the second row of $S_{q,\mu}$ to the other columns of $\Psi_{q,\mu}$ equal, obviously, to zero.

The multiplications 
of the third line of $S_{q,\mu}$ to the first three columns of $\Psi_{q,\mu}$ give us 
$$
0\cdot \Psi_{q,\mu} ^{1,1} + d^*_{q-2} \Psi_{q,\mu} ^{2,1}  + 0 \cdot 0   + d_{q-2} \cdot 0= 
 d^*_{q-2} d_{q-1}^* \Phi_{q} = 0,
$$
$$ 
0 \cdot d_{q-1} \Phi_{q-1} + d^*_{q-2} \Psi_{q,\mu} ^{2,2} + 
0 \cdot \Phi_{q-2} d^*_{q-2} + 0 \cdot 0  = 
d^*_{q-2} d_{q-1}^* \Phi_{q} = 0
$$
$$
d^*_{q-2}  d_{q-2}\Phi_{q-2} + 0\cdot \Psi_{q,\mu} ^{3,3}
+ d_{q-3} d^*_{q-3}
\Phi_{q-2}   + 0\cdot 0 = 
\Delta_{q-2}\Phi_{q-2} =I_{k_{q-2}}, 
$$
and the multiplications 
of the third line of $S_{q,\mu}$ to the other columns of $\Psi_{q,\mu}$ equal, obviously, to zero.

Next, for $q\geq 3$ and all $j$ with $4\leq j \leq q+1$,   the multiplication 
of $j$-th line of $S_{q,\mu}$ to $j$-th column of $\Psi_{q,\mu}$ equals to 
(here $d_{-1} \equiv 0$)
$$
d^*_{q-j+1}  d_{q-j+1}\Phi_{q-j+1} + 0\cdot 0 + d_{q-j} d^*_{q-j} \Phi_{q-j+1}   = \Delta_{q-j+1}\Phi_{q-j+1} 
=I_{k_{q-j+1}}.
$$
It follows from \eqref{eq.deRham.1}, \eqref{eq.deRham.1mu},  that 
the multiplications 
of the other lines of $S_{q,\mu}$ to the other columns of $\Psi_{q,\mu}$ equal, obviously, to zero.

These calculations mean that $\Psi_{q,\mu} $ is a right fundamental solution to $S_{q,\mu}$. As 
formally adjoint to a right fundamental solution for a differential operator is a left fundamental solution 
to its formally adjoint, then 
 \begin{equation} \label{eq.RFR.const}
    S_{q,\mu} \Psi_{q,\mu} =I = 
   (\Psi_{q,\mu})^* S_{q,\mu}^*  =  (\Psi_{q,\mu})^* S_{q,\mu} ,
\end{equation}
because operator $S_{q,\mu}$ is formally self-adjoint.  
%

Finally, \eqref{eq.deRham.1} and the convolution type of the 
fundamental solutions $\Phi _{\Delta^p_q}$ yield  
\begin{equation} \label{eq.deRham.21}
d_q \Phi_q =  \Phi_{q+1} d_q , \, d^* _{q-1} \Phi_q =  \Phi_{q-1} d^*_{q-1} \mbox{ on } C^\infty_0 ({\mathbb R}^n,\Lambda^q),
\end{equation}
\begin{equation} \label{eq.deRham.22}
d^*_q d_q \Phi_{\Delta^2_q} = 
\Phi_{\Delta^2_q} d^*_q d_q  , \, d_{q-1}d^* _{q-1} \Phi_{\Delta^2_q} =  \Phi_{\Delta^2_q} 
 d_{q-1} d^*_{q-1} \mbox{ on } C^\infty_0 ({\mathbb R}^n,\Lambda^q),
\end{equation}
\begin{equation} \label{eq.deRham.3}
(\Phi_{\Delta^p_q} u,v)_{L^2 ({\mathbb R}^n,\Lambda^q)} =  ( u, \Phi_{\Delta^p_q} v)_{L^2 
({\mathbb R}^n,\Lambda^q)} \mbox{ for all } u,v\in C^\infty_0 ({\mathbb R}^n,\Lambda^q), \, p\in \mathbb N.
\end{equation}
Thus $(\Psi_{q,\mu})^* = \Psi_{q,\mu}$, and hence it is a bilateral fundamental solution for $S_{q,\mu}$.
\end{proof}

Let's discuss the construction of fundamental solutions to the generalized Stokes' operators 
with variable coefficients under reasonable assumptions. 

\begin{thm} \label{t.40} 
Let $q\geq 1$, $X$ be a bounded domain in ${\mathbb R}^n$, $ \Delta _{j,\mu} = 0$ for all $0\leq j \leq q-1$ and the self-adjoint 
matrices ${\mathcal M}_j, \tilde {\mathcal M}_j$, $j_0\leq j\leq q$, are positive and $C^\infty$-smooth on 
a neighbourhod of the compact 
$\overline X$. Then there is a bilateral fundamental solution $\Psi_{q,\mu}$ for  $S_{q,\mu}$ on $X$. 
Moreover if matrices  
     ${\mathcal M}_q$, 
$\widetilde {\mathcal M}_q$ are real analytic on 
a neighbourhod of the compact  $\overline X$, 
then $S_{q,\mu}$ admits a pseudo-differential (in particular, regular and continuable) bilateral fundamental solution on $X$.  
\end{thm}

\begin{proof} We may assume that ${\mathcal M}_q, \widetilde {\mathcal M}_q$  are positive and $C^\infty$-smooth on  
the closure $\overline Y$ of a domain  $Y$, containing   $\overline X$ and having 
$C^\infty$-smooth boundary. Then
strongly elliptic, formally non-negative and formally self-adjoint second order operator 
$$
\Delta_{q,\mu} = A_{q,\mu} ^* A_{q,\mu}, \,\, A_{q,\mu} = 
\left(
\begin{array}{lll} 
\sqrt{{\mathcal M}_q (x) } \, d_q \\ 
\sqrt{\tilde {\mathcal M}_q (x)}  \, d^*_{q-1}
\end{array}
\right),
$$ 
where 
$\sqrt{{\mathcal M}}$ is the self-adjoint positive square root of a self-adjoint positive matrix, 
has a bilateral fundamental solution on $Y$. Indeed, consider the Dirichlet problem: 
given $f\in H^{-1} (Y,\Lambda^q)$ find $u  \in H^{1}_0 (Y,\Lambda^q)$ such that 
\begin{equation} \label{eq.Dir}
({\mathcal M}_q d_q u, d_q v)_{L^2 (X, \Lambda^{q+1})}+ (\tilde {\mathcal M}_{q} d^*_{q-1} u, 
d^*_{q-1} v)_{L^2 (X, \Lambda^{q-1})} \mbox{ for all } v \in H^{1}_0 (X,\Lambda^q),
\end{equation}
where $H^{1} (X)$ is the Sobolev space on $X$, $H^{1}_0 (X)$ is the closure of 
$C^\infty_0 (X)$ in  $H^{1} (X)$ and $H^{-1} (X)$ is the dual to $H^{1}_0 (X)$. 
By classical arguments, the problem is formally self-adjoint, formally non-negative and 
has the Fredholm property. Then it admits the Hodge parametrix, see, for instance, \cite[Ch. 10]{WRL}, \cite{SShTa}.
The Green function is granted, if the space $S_{A_{q,\mu}} (Y)\cap H^1_0 (Y,\Lambda^q)$ of solutions to the related homogeneous problem is trivial. 
However, under the assumtion on matrices ${\mathcal M}_q, \widetilde {\mathcal M}_q$ any element  $u$ of this space 
satisfies $d_q u =0$ and  
$d^*_{q-1} u =0$ in $Y$, and hence, according to \eqref{eq.Laplace.0}, \eqref{eq.Laplace.1}, its coefficients are 
harmonic in $Y$. Therefore the space is trivial becsuse of the Uniqueness Theorem 
for the Dirichlet Problem for the Laplace operator. Thus, we may take the Green function of the Dirichlet 
\eqref{eq.Dir} as the bilateral fundamental solution to the operator $\Delta_{q,\mu}$ on $Y$.  

Now, if   $\Phi_{q,\mu}$ is the bilateral fundamental solution to the operator $\Delta_{q,\mu}$ on a neighbourhood $Y$ 
of the compact  $\overline X$, and $\chi \in C^{\infty}_0 (Y)$ is a function that equals to $1$ a neighbourhood  
of the compact  $\overline X$, then the Schwartz kernels of the operators  
$\chi \Phi_ {q,\mu}$, 
$\chi \widetilde {\mathcal M}_q d^*_{q-1}\Phi_{q}$ and $ \chi  
\widetilde {\mathcal M}_q d_{q-1}^* \Phi_{q,\mu}$ are properly supported on $Y$. Then the compositions    
$\Phi_ {q,\mu} \chi \Phi_ {q,\mu} $, $ \Phi_{q-1} 
\chi \widetilde {\mathcal M}_q d_{q-1}^* \Phi_{q}$ 
and $ \Phi_{q-1} 
\chi \widetilde {\mathcal M}_q d_{q-1}^* \Phi_{q,\mu}$, used in the proof of the theorem 
are well-defined on $Y$, Moreover the interior a priori estimates for the elliptic differential 
operators $\Delta_q$ and $\Delta_{q,\mu}$ imply that   
 the compositions map  
$C_0^\infty (X, \Lambda^q)$ to 
$C^\infty (Y, \Lambda^q)$. Besides, as 
 $\chi \in C^{\infty}_0 (Y)$, then 
\begin{equation} \label{eq.chi}
\Delta_{q,\mu} \Phi _{q,\mu} \chi  = \chi  
\mbox{ on } C^\infty (Y, \Lambda^q), \,\, 
\Delta_{q} \Phi _{q} \chi = \chi \mbox{ on } C^\infty  (Y, \Lambda^q).
\end{equation} 

\begin{rem} \label{r.chi}
To continue the proof we note the following 

1)
all the operator identities in the proof of Lemma \ref{l.3} and Theorem \ref{t.40} are verified on  $X$, i.e., 
we calculate 
$$
\langle \Psi^{(r)}_{q,\mu} S_{q,\mu} \varphi ,  \psi\rangle, 
\langle S_{q,\mu}  \Psi_{q,\mu} ^{(r)}\varphi ,  \psi\rangle \mbox{  for } \varphi, \psi \in 
{\mathbf C}_{0,q}^\infty (X)  = \oplus_{j=0}^q C^\infty_{0} (X,\Lambda^j);
$$
 
2) the supports of the expressions, containing derivatives of the function $\chi$,  lie outside of  $\overline  X$. \end{rem}

\begin{lem} \label{l.3} 
Under the hypothesis of Theorem \ref{t.40}, let    
 $\Phi_{q,\mu}$ be a bilateral fundamental solution of the operator 
$\Delta_{q,\mu}$ in a neighbourhood $Y$ of the compact $\overline X$, and $\chi \in C^{\infty}_0 (Y)$.
Then

1) one of the right fundamental solutions to the operator $S_{q,\mu}$ on $X$ is given by  
three-diagonal $((q+1)\times (q+1))$-block matrix  $\Psi^{(r)}_{q,\mu}$ with the following  
block components:
$$
\Psi_{q,\mu} ^{(r),1,1} = d_q^* {\mathcal M}_q d_q \Phi_{q,\mu} \chi  \Phi_{q,\mu}, \, 
\Psi_{q,\mu} ^{(r),1,2}  = d_{q-1} \Phi_{q-1}, \, q\geq 1, 
$$
$$
\Psi_{q,\mu} ^{(r),2,1} =
\widetilde {\mathcal M}_q d_{q-1}^* \Phi_{q,\mu} , \, 
\Psi_{q,\mu} ^{(r),2,2} = -
 \widetilde {\mathcal M}_q , \, q=1,
$$
 and, for $q\geq 2$,
$$
\Psi_{q,\mu} ^{(r),2,1} =
d_{q-1}^*
d_{q-1} \Phi_{q-1}  \chi  
\widetilde {\mathcal M}_q d_{q-1}^* \Phi_{q,\mu} , \, 
\Psi_{q,\mu} ^{(r),2,2} =
- 
d_{q-1}^*
d_{q-1}  \Phi_{q-1} 
\chi \widetilde {\mathcal M}_q  d_{q-1}^* 
\Phi_{q} d_{q-1}, 
$$
all the other components of the matrix  $\Psi^{(r)}_{q,\mu}$ are given by formulae \eqref{eq.q.geq.2.1},  
\eqref{eq.q.geq.2.2};

2)  $\Psi^{(l)}_{q,\mu} = 
(\Psi^{(r)}_{q,\mu})^*$ is a left fundamental solution to  
 $S_{q,\mu}$. 
\end{lem}

\begin{proof} 
If $q=1$, then, by    
 \eqref{eq.deRham.1mu},  
\eqref{eq.chi} and Remark \ref{r.chi}, we obtain on $X$: 
$$
\Delta_{1,\mu} d_1^* {\mathcal M}_1 d_1 \Phi_{1,\mu}  \chi \Phi_{1,\mu}  + d_0 
\widetilde {\mathcal M}_1 d_0^* \Phi_{1,\mu} = d_1^* {\mathcal M}_1 d_1 \chi \Phi_{1,\mu} + d_0
\widetilde {\mathcal M}_1 d_0^* \Phi_{1,\mu}  
= I, $$
$$
d^*_0 d_0 \Phi_{0}  = I, \, 
d_0^*  d_1^* {\mathcal M}_1 d_1 \Phi_{1,\mu} \chi \Phi_{1,\mu} =  0,
$$
$$
\Delta _{1,\mu} 
 d_0 \Phi_{0} 
 -  d_0 \widetilde {\mathcal M}_1  = d_0 \widetilde {\mathcal M}_1 d^*_0 d_0 \Phi_{0} 
- d_0 \widetilde {\mathcal M}_1 = d_0 \widetilde {\mathcal M}_1  - d_0 \widetilde {\mathcal M}_1  = 0.
$$
Thus  $\Psi_{1,\mu}^{(r)} $ is a right fundamental solution to  $S_{1,\mu}$ on $X$.

In order to prove the statement of the lemma for $q\geq 2$, we have to repeat 
the related part of the proof of Theorem \ref{t.const}. However iin this situation 
the components of  $\Psi_{q,\mu} ^{(r)}$ coincide with the components 
of the matrix $\Psi_{q,\mu}$ from Theorem  \ref{t.const}  for all the pairs $(i,j)$, except 
$(1,1)$, $(2,1)$ and $(2,2)$. Therefore taking into the accpunt the structure of Stokes' operators, 
we have to multiply  the first three lines 
of $S_{q,\mu}$ on the first two columns of the matrix  $\Psi_{q,\mu}$. 

Thus, \eqref{eq.deRham.1}, 
\eqref{eq.deRham.1mu}, \eqref{eq.deRham.21}, 
\eqref{eq.chi} and Remark \ref{r.chi}  yield that the multiplication of the first line 
of $S_{q,\mu}$ on the first two columns of $\Psi^{(r)}_{q,\mu}$ gives us on  $X$: 
$$
\Delta_{q,\mu} \Psi_{q,\mu} ^{(r),1,1} + d_{q-1} \Psi_{q,\mu} ^{(r),2,1} + 0\cdot 0 = 
d_q^* {\mathcal M}_q d_q   \chi \Phi_{q,\mu} +  
d_{q-1} 
\Delta_{q-1} \Phi_{q-1} 
\chi \widetilde {\mathcal M}_q d_{q-1}^* \Phi_{q,\mu} = 
$$
$$
d_q^* {\mathcal M}_q d_q   \chi\Phi_{q,\mu} +  
d_{q-1}  \chi \widetilde {\mathcal M}_q d_{q-1}^* \Phi_{q,\mu} = 
\Delta_{q,\mu} \Phi_{q,\mu} =I_{k_q}, 
$$
$$
\Delta_{q,\mu}  \Psi_{q,\mu} ^{(r),1,2} +   d_{q-1} \Psi_{q,\mu} ^{(r),2,2} + 
0\cdot \Phi_{q-2}  d ^{*}_{q-2} + 0 \cdot 0 = 
$$
$$
d_{q-1} \widetilde {\mathcal M}_q d_{q-1}^* d_{q-1}  \Phi_{q-1} - 
 d_{q-1} (d_{q-1}^*
 d_{q-1} \Phi_{q-1}  
\chi  \widetilde {\mathcal M}_q  d_{q-1}^* \Phi_{q} d_{q-1} ) + 
0\cdot \Phi_{q-2}  d ^{*}_{q-2} = 
$$
$$
d_{q-1} \widetilde {\mathcal M}_q d_{q-1}^* d_{q-1}  \Phi_{q-1} - d_{q-1} \Delta_{q-1}\Phi_{q-1} 
\chi \widetilde {\mathcal M}_q  d_{q-1}^* \Phi_{q} d_{q-1}  =  
$$
$$
d_{q-1} \widetilde {\mathcal M}_q d_{q-1}^*  \Phi_{q} d_{q-1}    - d_{q-1} \chi \widetilde {\mathcal M}_q d_{q-1}^* \Phi_{q}  d_{q-1}  =0. 
$$
For the same reasons, the multiplication of the second line of 
$S_{q,\mu}$ on the first two columns of  
$\Psi^{(r)}_{q,\mu}$ gives us on $X$:
$$
 d^*_{q-1} \Psi_{q} ^{(r),1,1}  + 0 \cdot \Psi_{q} ^{(r),2,1}+ 0 \cdot 0  = 
 d^*_{q-1}  d_q^* {\mathcal M}_q d_q \Phi_{q,\mu}  \chi \Phi_{q,\mu} =0 , 
$$
$$
d^*_{q-1}  \Psi_{q} ^{(r),1,2} + 0\cdot \Psi_{q} ^{(r),2,2} + d_{q-2}   d^*_{q-2} \Phi_{q-1} + 0\cdot 0 
= \Delta_{q-1}\Phi_{q-1} =I_{k_{q-1}}.
$$
Finally,  the multiplication of the third line 
of $S_{q,\mu}$ on the first two columns of  
$\Psi^{(r)}_{q,\mu}$ gives us 
$$
0\cdot \Psi_{q,\mu} ^{(r),1,1} + d^*_{q-2} \Psi_{q,\mu} ^{(r),2,1}  + 0 \cdot 0   + d_{q-2} \cdot 0= 
 d^*_{q-2} d_{q-1}^* 
 d_{q-1} \Phi_{q-1} \chi  \widetilde {\mathcal M}_q d_{q-1}^* \Phi_{q,\mu} = 0,
$$
$$ 
0 \cdot \Psi_{q,\mu} ^{(r),1,2}+ d^*_{q-2} \Psi_{q,\mu} ^{(r),2,2} + 
0 \cdot \Phi_{q-2} d^*_{q-2} + 0 \cdot 0  =
$$
$$ 
d^*_{q-2} (- 
d_{q-1}^*
d_{q-1}  \Phi_{q-1} 
\chi \widetilde {\mathcal M}_q  d_{q-1}^* 
\Phi_{q} d_{q-1})
 = 0.
$$
This calculations mean that $\Psi_{q,\mu}^{(r)} $ is a right fundamental solution to 
the operator $S_{q,\mu}$. As we have seen in the proof of 
Theorem \ref{t.const}, 
  \begin{equation} \label{eq.RFR}
    S_{q,\mu} \Psi_{q,\mu}^{(r)} =I = 
   (\Psi^{(r)}_{q,\mu})^* S_{q,\mu}^*  = \Psi^{(l)}_{q,\mu} S_{q,\mu} ,
\end{equation}
because $S_{q,\mu}$ is formally self-adjoint. The lemma is proved. 
\end{proof}

\begin{rem} \label{r.commute}
A simpler right fundamental solution to $S_{q,\mu}$ was indicated on the symbolic level in \cite[Theorem 10]{ShPMi}
under the following assumption:
\begin{equation} \label{eq.commute}
d_{q-1} \tilde{\mathcal M}_q d_{q-2} \equiv 0 .
\end{equation}
Assumption \eqref{eq.commute} allows to consider matrices $\tilde {\mathcal M}_{q}$  with smooth 
entries, too. Indeed, if we treat the term $\tilde {\mathcal M}_{q} d _{q-2} u$ as in 
\eqref{eq.multip}, then 
$$
d _{q-1} \tilde {\mathcal M}_{q} d _{q-2}u = d_{q-1}\sum_{\sharp I = q-1} (\star (d_{q-2}u \wedge \star \tilde 
{\mathcal M}^{(I)}_q ))dx_I = 
$$
$$
\sum_{j=1}^n \sum_{\sharp I = q-1} \sum_{\sharp J = q-1}  (d_{q-2}u)_J  \frac{\partial \tilde {\mathcal M}^{(J)}_{q,I} }{\partial x_j} dx_j \wedge dx_I  +
$$
$$
\sum_{j=1}^n \sum_{\sharp I = q-1} \sum_{\sharp J = q-1}   \frac{\partial 
(d_{q-2}u)_J }{\partial x_j} \tilde  {\mathcal M}^{(I)}_J  dx_j \wedge dx_I = 
$$
$$
\sum_{\sharp J = q-1}   (d_{q-2}u)_J  d_{q-1}\tilde {\mathcal M}^{(J)}_q  + 
\tilde  {\mathcal M} (d_{q-1} (d_{q-2}u))  =\sum_{\sharp J = q-1}   (d_{q-2}u)_J  d_{q-1}\tilde {\mathcal M}^{(J)}_q.
$$
Thus,  \eqref{eq.commute} is fulfilled 
if only if 
 the forms $\tilde {\mathcal M}^{(J)}_q $ are closed in $X$ 
 for all  $J$  with  $\sharp J = q-1$.
\end{rem}

In order to construct a two-sided fundamental solution $\Psi_{q,\mu}$ to $S_{q,\mu}$ we have to find 
a smoothing operator $H$, such that 
\begin{equation}\label{eq.bilateral.1q}
S_{q,\mu} (\Psi_{q,\mu}^{(r)} + H) = I, \,\,  (\Psi_{q,\mu}^{(r)} + H)  S_{q,\mu} = I. 
\end{equation}
With this purpose we again envoke Remark   \ref{r.chi}.

If $q=1$ then 
\begin{equation}\label{eq.bilateral.2}
\Psi_{1,\mu}^{(r)}  S_{1,\mu} = 
\left( 
\begin{array}{lll}
d_1^* {\mathcal M}_1 d_1  \Phi_{1,\mu} \chi + d_0 \Phi_0 d_0^* & 
 d_1^* {\mathcal M}_1 d_1 \Phi_{1,\mu} 
 \chi\Phi_{1,\mu} d_0 \\
0 &  \widetilde {\mathcal M}_1 d_0^* \Phi_{1,\mu} d_0\\
\end{array}
\right) = I- {\mathcal A}_{1,\mu} , 
\end{equation} 
where 
\begin{equation} \label{eq.A1}
{\mathcal A}_{1,\mu} = 
\left( 
\begin{array}{lll} 
d_1^* d_1  \Phi_1 - d_1^* {\mathcal M}_1 d_1  \Phi_{1,\mu} & 
 - d_1^* {\mathcal M}_1 d_1 \Phi_{1,\mu} \chi\Phi_{1,\mu} d_0\\
0 & d^*_0d_0 \Phi_0 - \widetilde {\mathcal M}_1 d_0^* \Phi_{1,\mu} d_0 \\
\end{array}
\right), 
\end{equation}
because the multiplication on the function  $\chi$ acts as the identity operator on the space of functions 
with compact supports in $X$.

On the other hand, by 
\eqref{eq.deRham.1}, \eqref{eq.deRham.1mu}, \eqref{eq.chi}, we obtain on $X$ the components 
of the matrix  $S_{1,\mu} {\mathcal A}_{1,\mu}$:
\begin{equation}\label{eq.bilateral.2.11}
\begin{array}{ccccc}
\Delta_{1,\mu} a_{11} + d_{0} a_{21} = d_1^* {\mathcal M}_1 d_1 \Delta_1 \Phi_1  - d_1^* {\mathcal M}_1 d_1 =0 , \\ [.1cm]
 \Delta_{1,\mu} a_{12} + d_{0} a_{22} = 
-\Delta_{1,\mu} d_1^* {\mathcal M}_1 d_1  \Phi_{1,\mu} \chi \Phi_{1,\mu} d_0 + d_0 
(d^*_0 d_0 \Phi_0 - \widetilde {\mathcal M}_1 d_0^* \Phi_{1,\mu} d_0 ) = \\[.1cm]
 - d_1^* {\mathcal M}_1 d_1  \chi \Phi_{1,\mu} d_0 - d_0\widetilde {\mathcal M}_1 d_0^* \Phi_{1,\mu} d_0 + d_0 
 = -d_0 + d_0 = 0,\\[.1cm]
d_{0}^* a_{11} + 0 \cdot a_{21} =  d_0^ *(d_1^*  d_1 \Phi_1  - d_1^* {\mathcal M}_1 d_1  \Phi_{1,\mu} ) + 0 \cdot 0 =0,\\[.1cm]
d_{0}^* a_{12} + 0 \cdot a_{22} = 
d_0^*  (- d_1^* {\mathcal M}_1 d_1 \Phi_{1,\mu} \chi \Phi_{1,\mu} d_0) + 0 \cdot (d^*_0d_0 \Phi _0 - \widetilde {\mathcal M}_1 d_0^* \Phi_{1,\mu} d_0)= 0,\\[.1cm]
	\end{array}
\end{equation}
where  $a_{ij}$  arethe components 
of the matrix ${\mathcal A}_{1,\mu}$. 
Besides, according to 
\eqref{eq.deRham.0}, 
\eqref{eq.deRham.1} on $X$ we have additional relations for  $a_{ij}$, that we will use later:
\begin{equation}\label{eq.bilateral.2.12}
\begin{array}{ccccc} 
a_{11} \Delta_{1,\mu}  = (d_1^*  d_1 \Phi_1  - d_1^* {\mathcal M}_1 d_1  \Phi_{1,\mu} )\Delta_{1,\mu} =
 \Delta_1 \Phi_1  d^*_1 {\mathcal M}_1 d_1  - d_1^* {\mathcal M}_1 d_1 =0, \\[.1cm]
 d_0 \widetilde {\mathcal M}_1 \Delta_0 = d_0 \widetilde {\mathcal M}_1 d^*_0d_0 = 
\Delta_{1,\mu} d_0, \\[.1cm]
- a_{12}(\widetilde {\mathcal M}_1 \Delta_0)^2 =   
d_1^* {\mathcal M}_1 d_1 \Phi_{1,\mu} \chi \Phi_{1,\mu} d_0 (\widetilde {\mathcal M}_1 \Delta_0)^2 = 
d_1^* {\mathcal M}_1 d_1 \Phi_{1,\mu} \chi  \Phi_{1,\mu} \Delta_{1,\mu}^2 d_0  = 0, \\[.1cm]
a_{22}\widetilde {\mathcal M}_1 \Delta_0 = 
(d^*_0d_0 \Phi_0 - \widetilde {\mathcal M}_1 d_0^* \Phi_{1,\mu} d_0) \widetilde {\mathcal M}_1 \Delta_0 = 
\widetilde {\mathcal M}_1 (\Delta_0  -  d_0^* \Phi_{1,\mu}  \Delta_{1,\mu}  d_0 ) 
=0.
	\end{array}
\end{equation}

If $q\geq 2$ then, by \eqref{eq.deRham.1}, 
\eqref{eq.deRham.1mu}, 
\eqref{eq.deRham.21}, we have 
$\Psi_{q,\mu}^{(r)}  S_{q,\mu} = {\mathcal B}_{q,\mu} $, where the block-entries $b_{ij}$ 
of the block-matrix ${\mathcal B}_{q,\mu}$ are as follows:
\begin{equation*} 
b_{11} =  \Psi_{q,\mu}^{(r),1,1}  \Delta_{q,\mu} + d_{q-1} \Phi_{q-1} d_{q-1}^* \, 
=
 d_q^* {\mathcal M}_q d_q  \Phi_{q,\mu}  + d_{q-1} \Phi_{q-1} d_{q-1}^* , 
\end{equation*} 
\begin{equation*} 
b_{11} =  \Psi_{q,\mu}^{(r),1,1}  \Delta_{q,\mu} +\Psi_{q,\mu}^{(r),1,2} d_{q-1}^* \, 
=
 d_q^* {\mathcal M}_q d_q  \Phi_{q,\mu} \chi + d_{q-1} \Phi_{q-1} d_{q-1}^* , 
\end{equation*} 
\begin{equation*} 
b_{12} = \Psi_{q,\mu}^{(r),1,1} d_{q-1} + \Psi_{q,\mu}^{(r),1,2} \cdot 0 + 0 \cdot 0   
=d_q^* {\mathcal M}_q d_q  \Phi_{q,\mu} \chi  \Phi_{q,\mu} d_{q-1},  
\end{equation*} 
$$
 b_{13} =  \Psi_{q,\mu}^{(r),1,1} \cdot 0 + d_{q-1} \Phi_{q-1} d_{q-2} + 0 \cdot d^*_{q-1} + 0 \cdot 0 
 = 0,  
$$
$$
b_{21} = \Psi_{q,\mu}^{(r),2,1} \Delta_{q,\mu} +  \Psi_{q,\mu}^{(r),2,2} d^*_{q-1} + 0 \cdot 0= 
$$
$$
d_{q-1}^* d_{q-1} \Phi_{q-1} 
\chi\widetilde {\mathcal M}_q d_{q-1}^* \Phi_{q,\mu} 
\Delta_{q,\mu}  - d_{q-1}^* d_{q-1}\Phi_{q-1}  
\chi \widetilde {\mathcal M}_q d_{q-1}^* d_{q-1} \Phi_{q-1}
   d^*_{q-1} =
$$
$$
d_{q-1}^* d_{q-1} \Phi_{q-1} 
\chi\widetilde {\mathcal M}_q d_{q-1}^*   - d_{q-1}^* d_{q-1} \Phi_{q-1} 
\chi \widetilde {\mathcal M}_q d_{q-1}^* =0, 
$$
$$ b_{22} = \Psi_{q,\mu}^{(r),2,1} d_{q-1}+ 0 \cdot \Psi_{q,\mu}^{(r),2,2} + d_{q-2} \Phi_{q-2} d_{q-2}^* =  
$$
$$
d_{q-1}^* d_{q-1} \Phi_{q-1} 
\chi 
\widetilde {\mathcal M}_q d_{q-1}^* \Phi_{q,\mu} d_{q-1} + 
d_{q-2} \Phi_{q-2} d_{q-2}^*,  
$$
$$
b_{23} =  \Psi_{q,\mu}^{(r),2,1} \cdot 0 +  \Psi_{q,\mu}^{(r),2,2}d_{q-2} + d_{q-2}\Phi_{q-2} \cdot 0 + 0
\cdot  d_{q-1}^* + 0 \cdot 0 =   
$$
$$
- d_{q-1}^* d_{q-1}
\Phi_{q-1}
\chi \widetilde {\mathcal M}_q d_{q-1}^*\Phi_{q} d_{q-1} d_{q-2} =0, \,
$$
$$
 b_{j,j} = I_{k_{q-j+1}}, 3\leq j \leq q+1,   \, 
b_{i,j} = 0, 3\leq i,j \leq q+1, i\ne j.
$$
or, in other form, induced by \eqref{eq.deRham.0} and \eqref{eq.deRham.21}, 
\begin{equation}\label{eq.bilateral.2q} 
\Psi_{q,\mu}^{(r)}  S_{q,\mu} = I- {\mathcal A}_{q,\mu}
\end{equation}
where  the block-entries $a_{ij}$ 
of the block-matrix ${\mathcal A}_{q,\mu}$ are given by
\begin{equation} \label{eq.Aq}
a_{11} =
d_q^* d_q  \Phi_q - d_q^* {\mathcal M}_q d_q  \Phi_{q,\mu}, \, 
a_{12} =   -d_q^* {\mathcal M}_q d_q  \Phi_{q,\mu} \chi \Phi_{q,\mu} d_{q-1}, \, a_{13} =0, 
\end{equation}
$$
 a_{21} = 0, \, a_{22} = d^*_{q-1}d_{q-1} \Phi_{q-1} - 
d_{q-1}^* d_{q-1} \Phi_{q-1} 
\chi 
\widetilde {\mathcal M}_q d_{q-1}^* \Phi_{q,\mu} d_{q-1} ,
\, a_{23} =0,
$$
$$
a_{i,j} = 0, 3\leq i,j \leq q+1,
$$
because 
the multiplication on the function  $\chi$ acts as the identity operator on the space of functions 
with compact supports in $X$.

On the the other hand, 
according to \eqref{eq.deRham.1}, 
\eqref{eq.deRham.1mu} 
and Remark \ref{r.chi}, we obtain on $X$ the components of the matrix $S_{q,\mu} {\mathcal A}_{q,\mu}$:   
\begin{equation}\label{eq.bilateral.2q.11}
\begin{array}{ccccc}
\Delta_{q,\mu} a_{11} + d_{q-1} a_{21} =
\\ [.1cm]
\Delta_{q,\mu} (d_q^*  d_q \Phi_q  - d_q^* {\mathcal M}_q d_q  \Phi_{q,\mu} ) + d_{q-1} \cdot 0 =  
 d_q^* {\mathcal M}_q d_q \Delta_q \Phi_q  - d_q^* {\mathcal M}_q d_q =0 , \\ [.1cm]
d_{q-1} d_{q-1}^* d_{q-1} \Phi_{q-1} =d_{q-1} d^*_{q-1} d_{q-1} \Phi_{q-1} =d_{q-1}\Delta_{q-1} 
 \Phi_{q-1} = d_{q-1}, \\ [.1cm]
\Delta_{q,\mu} a_{12} + d_{q-1} a_{22} =
-\Delta_{q,\mu} d_q^* {\mathcal M}_q d_q  \Phi_{q,\mu} \chi\Phi_{q,\mu} d_{q-1} + \\ [.1cm]
d_{q-1} 
(d^*_{q-1} d_{q-1} \Phi_{q-1} - d_{q-1}^* d_{q-1} \Phi_{q-1} \chi  \widetilde {\mathcal M}_q d_{q-1}^* \Phi_{q,\mu} d_{q-1} ) = 
\\ [.1cm]
 - d_q^* {\mathcal M}_q d_q   \Phi_{q,\mu} d_{q-1} - 
d_{q-1}
\chi  
\widetilde {\mathcal M}_q d_{q-1}^* \Phi_{q,\mu} 
d_{q-1} + d_{q-1} = \\ [.1cm]
 = - \Delta_{q,\mu}  \Phi_{q,\mu}  d_{q-1} + d_{q-1} = 0.\\ [.1cm]
d_{q-1}^* a_{11} + 0 \cdot a_{21} = 
 d_{q-1}^ *(d_q^*  d_q \Phi_q  - d_q^* {\mathcal M}_q d_q  \Phi_{q,\mu} ) + 0 \cdot 0 =0,\\[.1cm]
d_{q-1}^* a_{12} + 0 \cdot a_{22} = 
d_{q-1}^* (-d_q^* {\mathcal M}_q d_q  \Phi_{q,\mu} \chi \Phi_{q,\mu} d_{q-1}) = 0.\\ [.1cm]
	\end{array}
\end{equation}
Then, taking into the account \eqref{eq.deRham.21}, we obtain on $X$ additional relations 
for  $a_{ij}$: 
\begin{equation}\label{eq.bilateral.2q.12}
\begin{array}{ccccc} 
a_{11} \Delta_{q,\mu}  = (d_q^*  d_q \Phi_q  - d_q^* {\mathcal M}_q d_q  \Phi_{q,\mu} )\Delta_{q,\mu} =
\Phi_q \Delta_q d^*_q {\mathcal M}_q d_q  - d_q^* {\mathcal M}_q d_{q} =0, \\ [.1cm]
d_{q-1} (\widetilde {\mathcal M}_q d^*_{q-1} d_{q-1}+ d_{q-2}d^*_{q-2})  = 
d_{q-1} \widetilde {\mathcal M}_q d^*_{q-1}d_{q-1}   
= \Delta_{q,\mu} d_{q-1}, 
\\ [.1cm]
d_{q-1} (\widetilde {\mathcal M}_q d^*_{q-1} d_{q-1}+ d_{q-2}d^*_{q-2})^2 = 
\Delta_{q,\mu} \Delta_{q,\mu} d_{q-1},\\ [.1cm]
- a_{12}(\widetilde {\mathcal M}_q d^*_{q-1} d_{q-1}+ d_{q-2}d^*_{q-2})^2 = 
\\ [.1cm]
d_q^* {\mathcal M}_q d_q  \Phi_{q,\mu} \chi
\Phi_{q,\mu} 
d_{q-1} (\widetilde {\mathcal M}_q d^*_{q-1} d_{q-1}+ d_{q-2}d^*_{q-2})^2 
= d_q^* {\mathcal M}_q d_q   d_{q-1} =0,\\ [.1cm]
a_{22} (\widetilde {\mathcal M}_{q} d_{q-1}^* d_{q-1} + d_{q-2}d^*_{q-2})= 
\\ [.1cm]
(d^*_{q-1}d_{q-1} \Phi_{q-1} - 
d_{q-1}^*d_{q-1} \Phi_{q-1} \chi\widetilde {\mathcal M}_q d_{q-1}^* \Phi_{q,\mu} d_{q-1}) 
(\widetilde {\mathcal M}_q d_{q-1}^* d_{q-1} + d_{q-2}d^*_{q-2}) = 
\\ [.1cm]
d^*_{q-1}d_{q-1} \Phi_{q-1} \widetilde {\mathcal M}_q d_{q-1} ^* d_{q-1} - 
d_{q-1}^* d_{q-1} \Phi_{q-1}\chi  \widetilde {\mathcal M}_q d_{q-1}^*d_{q-1} = 0.
	\end{array}
\end{equation}
again, because 
the multiplication on the function  $\chi$ acts as the identity operator on the space of functions 
with compact supports in $X$.

Now, according to formula   
\eqref {eq.bilateral.Aq}
 and the proof of Proposition \ref{p.0}, each block of the matrix 
${\mathcal A}_{q,\mu}$ is annihilated from the left by an (Petrovskii) elliptic operator with $C^\infty$-smooth 
coefficients. As the matrix $\widetilde {\mathcal M}_q$ is not-degenerate  then the operators  
$$
(\widetilde {\mathcal M}_q d^*_{q-1} d_{q-1}+ d_{q-2}d^*_{q-2}), \, 
(\widetilde {\mathcal M}_q d^*_{q-1} d_{q-1}+ d_{q-2}d^*_{q-2})^2, \, 
\Delta_{q,\mu}, \, q\geq 1,
$$
(Petrovskii) elliptic (here, as before, $d_{-1} \equiv 0$), then formulae  
 \eqref{eq.bilateral.2.12}, \eqref{eq.bilateral.2q.12} mean that each block of the matrix 
${\mathcal A}_{q,\mu}$  annihilates from the left by an (Petrovskii) elliptic operator with $C^\infty$-smooth 
coefficients.  Thus, these ellipic equations with respect to the variables 
 $x$ and $y$ provides the smoothness of the Schwartz kernel of the 
operator  ${\mathcal A}_{q,\mu}$  on  $X\times X$, 
i.e.  
the operator ${\mathcal A}_{q,\mu}$ is smoothing on  $X$.  

Therefore \eqref{eq.bilateral.2.11} and  \eqref{eq.bilateral.2q.11} imply
\begin{equation}\label{eq.bilateral.Aq}
S_{q,\mu} A_{q,\mu} \equiv 0 \mbox{ on } {\mathbf C}_q^\infty (D), \, q\geq 1. 
\end{equation}

Hence, formulae \eqref{eq.bilateral.2}, \eqref{eq.bilateral.2q}, imply that   \eqref{eq.bilateral.1q} is equivalent to 
\begin{equation}\label{eq.bilateral.4q}
S_{q,\mu} H = 0, \,\,   S_{q,\mu}  H ^* = {\mathcal A}^*_{q,\mu} , \, q\geq 1. 
\end{equation}
Therefore we may set $H = {\mathcal A}_{q,\mu} (\Psi_{q,\mu}^{(r)})^*$ that satisfies 
\eqref{eq.bilateral.4q} because $\Psi_{q,\mu}^{(r)}$ is a right fundamental 
solution to $S_{q,\mu}$, $q \geq 1$, and identity 
\eqref{eq.bilateral.Aq} holds true.


Finally, assume additionally that the coefficients of the operator $\Delta_{q,\mu}$ are real analytic and
recall that a differential operator  ${\mathfrak  A}$ has the Uniqueness Property 
in the small on $Y$ (or Property (A) from the paper by B. Malgrange \cite[page 334]{Mal56}), if  
\begin{itemize}
    \item[(US)] for any domain  $D$ from $Y$ a solution $u \in S_{{\mathfrak A}} (D)$ is identically zero, if 
 $u =0$ on an open subset  $G\subset U$.
\end{itemize} 
Then, By Petrovski Theorem,  
solutions to the operator $\Delta_{q,\mu}$ are real analytic and hence 
it satisfies the Uniqueness Property (US) on 
$Y$. Therefore   
$\Delta_{q,\mu}$ has a bilateral pseudo-differential 
fundamental solution  
$\Phi_{q,\mu}$ on $Y$ in this case, see. \cite[p. 349]{Mal56}, or, for instance,  \cite[Theorem 4.4.3 and Example 4.4.5]{Tark36}. 
As the Schwartz kernels of the operators  
$\chi \Phi_ {q,\mu}$, 
$\chi \widetilde {\mathcal M}_q d^*_{q-1}\Phi_{q}$ and $ \chi  
\widetilde {\mathcal M}_q d_{q-1}^* \Phi_{q,\mu}$, used in the construction 
of the fundamental solution $\Psi_{q,\mu}$ for $S_{q,\mu}$ are properly supperted on $Y$, then, by the well known 
theorem on the composition of pseudo-differential operators, see, for instance, \cite[Proposition 1.5.9]{Tark35},  
$\Psi_{q,\mu}$ is a pseudo-differential operator, too. 
In particular, both $\Phi_{q,\mu}$ and $\Psi_{q,\mu}$ are regular and continuable, see, for instance  
\cite[\S 1.5, pp. 78--82]{Tark36}. 
\end{proof}

These considerations result in a homotopy formula that is usually called the (second) Green 
formula for the operator $S_{q,\mu}$, see \cite[Theorem 2.4.8]{Tark35}. 

\begin{prop} \label{t.Green}
Let $D$ be a relatively compact Lipschitz domain in $X$ 
and $ \Delta _{j,\mu} = 0$ for all $j_0-1\leq j \leq q-1$. 
If the self-adjoint matrices ${\mathcal M}_j, \tilde {\mathcal M}_j$, $j_0\leq j\leq q$, are positive 
and $C^\infty$-smooth on $\overline X$, then
for any   $u \in  {\mathcal S} _{S_{q,\mu}} (D) \cap 
\Big(\big(\oplus_{j=j_0}^q C^1 (\overline D, \Lambda^j)\big) \oplus 
\big(\oplus_{j=0}^{j_0-1} C (\overline D, \Lambda^j)\big) \Big)$ 
we have 
\begin{equation}\label{eq.Green.Sqmu.2}
 - \int_{\partial D}
 {\mathcal G}_{S_{q,\mu}} ((\Psi^{(l)}_{q,\mu} (x,y))^*, u (y))  = 
\left\{ 
\begin{array}{lll}
u (x), & x \in D , \\
0, & x \not \in D.
\end{array}
\right.  
\end{equation}
\end{prop}

\begin{proof} Follows immediately from \cite[Theorem 2.4.8]{Tark35}, the (first) Green formula \eqref{eq.Green.Sqmu1} 
for $S_{q,\mu}$ and formulae \eqref{eq.Green.d}, \eqref{eq.GDelta}, \eqref{eq.GS} for Green's operators 
${\mathcal G}_{d_{j}}$, ${\mathcal G}_{\Delta_{j,\mu}}$, ${\mathcal G}_{S_{q,\mu}}$, 
 related to operators $d_j$, $\Delta_{j,\mu}$ and $S_{q,\mu}$, respectively.  
\end{proof}


{\bf Acknowledgements}. 
This work was supported by the Krasnoyarsk Mathematical Center and financed by the Ministry of Science and 
Higher Education of the Russian Federation (Agreement No. 075-02-2026-1314).


\begin{thebibliography}{XXXXXX}



\bibitem{deRham}
de Rham, G.  Differentiable manifolds. Forms, currents, harmonic forms. Grundlehren der mathematischen Wissenschaften. Vol. 266, 1984.
















\bibitem{Mal56}
{Malgrange, B.}  
\textit{Existence et approximation des solutions des \'equations aux d\'eriv\'ees 
partielles et des \'equations de convolution}, Annales de l'Institut Fourier, V. 6 
(1955/56),  271--355. 

\bibitem{Mal63}
{Malgrange, B.}  
\textit{Sur les syst\'emes differentiels \'a coefficients constants}, Colloq. internat. 
Centre nat. rech. scient., 117 (1963),  113--122. 
 

\bibitem{Mas97}
Maslennikova, V.N.  Partial Differential Equations, M., RUDN,  1997.

\bibitem{MeShT}
Mera A., Shlapunov A.A.,  Tarkhanov N. 
\textit{ Navier-Stokes Equations for 
Elliptic Complexes}, Journal of Siberian Federal University, Math. and Phys., {\bf 12}:1
(2019),  3--27.



\bibitem{Oz}
Oseen, C. W. \textit{\"Uber die Stokes'sche formel, und \"uber eine verwandte Aufgabe in der 
Hydrodynamik}, Arkiv f\"or matematik, astronomi och fysik, VI (29), 1910.


\bibitem{PaSh21}
Parfenov, A.A., Shlapunov, A.A.  \textit{On the stability phenomenon of the Navier-Stokes type 
Equations for Elliptic Complexes}, Complex Variables and Elliptic Equations, N. 6--7, 
 Volume 66 (2021), 1122--1150.








\bibitem{SShTa}
 Schulze B.-W.,  Shlapunov A.A., Tarkhanov N.  
\textit{ Green integrals on manifolds with cracks.}
Annals of Global Analysis and Geometry, Vol. 24, 2003, p. 131--160.


\bibitem{ShTaNS}
Shlapunov A.A., Tarkhanov N.  
An open mapping theorem for the Navier-Stokes type equations associated with the de 
Rham complex over ${\mathbb R}^n$, Siberian Electronic Math. Reports, 18:2 (2021), 1433--1466.



\bibitem{ShPMi}
Shlapunov A.A., Polkovnikov A.N., Mironov  V.L. \textit{ 
Maxwell's and Stokes's operators associated with elliptic differential complexes.} 
J. Math. Phys. 66, 011513 (2025). 

\bibitem{ShPo}
Shlapunov A. A., Polkovnikov, A. N.  \textit{ Generalized Navier--Stokes equations
associated with the Dolbeault complex.} 
J. of Math. Sciences, Vol. 293:3 (2025), 430-439. 

\bibitem{Tark35}
{Tarkhanov, N.} 
 {Complexes of differential operators},
 Kluwer Academic Publishers, Dordrecht, NL, 1995.

\bibitem{Tark36}
{Tarkhanov, N.} 
{The Cauchy Problem for Solutions of Elliptic Equations},
 Akademie-Verlag, Berlin, 1995.
\cite[\S 6.3.4]{Tark36}

\bibitem{Tark37}
{Tarkhanov, N.} 
 {The Analysis of Solutions of Elliptic Equations},
Kluwer Academic Publishers, Dordrecht, NL, 1997.

\bibitem{Tema79}
Temam, R.  
{Navier-Stokes Equations. Theory and Numerical Analysis},
  North Holland Publ. Comp., Amsterdam, 1979.
	
\bibitem{Vla} Vladimirov, V.S. Equations Of Mathematical Physics, Mir Publ., Moscow, 1984.




 

\bibitem{WRL}  Wloka, J. T., Rowley, B., and Lawruk,  B. Boundary Value Problems for Elliptic Systems,
Cambridge Univ. Press, Cambridge, 1995. 






\end{thebibliography}
\end{document}